\documentclass[12pt, twoside, leqno]{article}

\usepackage{amsmath,amsthm}
\usepackage{amssymb}
\usepackage{enumitem}
\usepackage{url}

\newtheorem{theorem}{Theorem}[section]
\newtheorem{lemma}{Lemma}[section]

\newtheorem{remark}{Remark}
\newtheorem*{concluding remark}{Concluding Remark}
\newtheorem{theoremA}{Theorem}

\begin{document}

\baselineskip=17pt

\title{On the average behavior of coefficients related to triple product L-functions}

\author{K VENKATASUBBAREDDY\\
School of Mathematics and Statistics\\ 
University of Hyderabad\\
Hyderabad\\
500046 Telangana, India\\
E-mail: 20mmpp02@uohyd.ac.in
\and 
AYYADURAI SANKARANARAYANAN\\
School of Mathematics and Statistics\\ 
University of Hyderabad\\
Hyderabad\\
500046 Telangana, India\\
E-mail: sank@uohyd.ac.in}

\date{}

\maketitle

\renewcommand{\thefootnote}{}
\renewcommand*{\thetheoremA}{\Alph{theoremA}}
\footnote{2010 AMS \emph{Mathematics subject classification.} Primary 11F30, 11F66.}


\renewcommand{\thefootnote}{\arabic{footnote}}
\setcounter{footnote}{0}


\begin{abstract}
    In this paper, we study the average behaviour of the coefficients of triple product L-functions and some related L-functions corresponding to normalized primitive holomorphic cusp form $f(z)$ of weight $k$ for the full modular group $SL(2,\ \mathbb{Z}).$ Here we call $f(z)$ a primitive cusp form if it is an eighenfunction of all Hecke operators simultaneously.
\end{abstract}

\section{Introduction}
For an even integer $k\geq 2$, denote $H_k^*$ the set of all normalized Hecke primitive cusp forms of weight $k$ for the full modular group $SL(2,\ \mathbb{Z})$. Throughout this paper we call the function $f(z)$ as a primitive cusp form if it is an eighenfunction of all Hecke operators simultaneously. It is known that $f(z)$ has a Fourier expansion at cusp $\infty$, write it as
\begin{equation*}
    f(z)=\sum_{n=1}^\infty c(n) e^{2\pi i n z}
\end{equation*}
for $\Im (z)>0$.\\
Rewrite the Fourier expansion as 
\begin{equation*}
    f(z)=\sum_{n=1}^\infty \lambda_f(n)n^{(k-1)/2}e^{2\pi i n z}
\end{equation*}
for $\Im (z)>0$, where $\lambda_f(n)=\frac{c(n)}{n^{(k-1)/2}}$.\\
Then by Deligne \cite{B}, we have, for any prime number $p$, there exists two complex numbers $\alpha_f(p)$ and $\beta_f(p)$, such that 
\begin{equation*}
    \alpha_f(p)\beta_f(p)=|\alpha_f(p)|=|\beta_f(p)|=1
\end{equation*} 
and 
\begin{equation*}
    \lambda_f(p)=\alpha_f(p)+\beta_f(p).
\end{equation*}
For a normalized primitive cusp form $f(z)$ of weight $k$, the triple product $L$-function $L(f\otimes f\otimes f)$ is defined as 
\begin{align*}
    L(f\otimes f\otimes f)&=\prod_p\bigg(1-\frac{\alpha_p^3}{p^s}\bigg)^{-1}\bigg(1-\frac{\alpha_p}{p^s}\bigg)^{-3}\bigg(1-\frac{\beta_p^3}{p^s}\bigg)^{-1}\bigg(1-\frac{\beta_p}{p^s}\bigg)^{-3}\\
    &=\sum_{n=1}^\infty\frac{\lambda_{f\otimes f\otimes f}(n)}{n^s}
\end{align*}
for $\Re (s)>1$. The $j^{th}$ symmetric power $L$-function attached to $f$ is defined by
\begin{equation}
        L({\rm{sym}}^jf,\ s)=\prod_p\prod_{m=0}^j(1-\alpha_p^{j-m}\beta_p^mp^{-s})^{-1}\label{E1}
\end{equation}
for $\Re (s)>1$. We may express it as a Derichlet series: for $\Re (s)>1$, 
\begin{align}
     L({\rm{sym}}^j f,\ s)&=\sum_{n=1}^\infty \frac{\lambda_{{\rm{sym}}^j f}(n)}{n^s}\nonumber\\
     &=\prod_p\bigg(1+\frac{\lambda_{{\rm{sym}}^j f}(p)}{p^s}+\ldots+\frac{\lambda_{{\rm{sym}}^j f}(p^k)}{p^{ks}}+\ldots\bigg).\label{E2}
\end{align}
It is well known that $\lambda_{{\rm{sym}}^j f}(n)$ is a real multiplicative function. The Rankin-Selberg $L$-function $L({\rm{sym}}^i f\otimes{\rm{sym}}^j f,\ s)$ attached to ${\rm{sym}}^i f$ and ${\rm{sym}}^j f$ is defined as
\begin{align}
    L({\rm{sym}}^if\otimes {\rm{sym}}^j f,\ s)&=\prod_p\prod_{m=0}^i \prod_{m{'}=0}^j\bigg(1-\frac{\alpha_p^{i-m}\beta_p^m\alpha_p^{j-m{'}}\beta_p^{m{'}}}{p^s}\bigg)^{-1}\nonumber\\
    &=\sum_{n=1}^\infty\frac{\lambda_{{\rm{sym}}^i f\otimes {\rm{sym}}^j f}(n)}{n^s}.\label{E3}
\end{align}
For $\Re (s)>1$, define
\begin{equation*}
    L_f(s)=\sum_{n=1}^\infty\frac{\lambda_{f\otimes f\otimes f}(n)^2}{n^s}
\end{equation*}
and for $\Re (s)>1$, define
\begin{equation*}
    D_f(s)=\sum_{n=1}^\infty\frac{\lambda_{{\rm{sym}}^2 f\otimes f}(n)^2}{n^s}.
\end{equation*}

\section{Theorems etc.}
Here, after $\epsilon$ and $\delta$ denote any small positive constants and implied constants will depend at most only on the form $f$ and $\epsilon.$\\
In the paper \cite{C}, the following theorems are established.

\begin{theoremA}
For any $\epsilon>0$, we have 
\[\sum_{n\leq x}\lambda_{f\otimes f\otimes f}(n)^2=xP(\log x)+O_{f,\ \epsilon}(x^{(175/181)+\epsilon})            \] where $P(t)$ is a polynomial of degree $4$.\label{TA}
\end{theoremA}
\begin{theoremA}
For any $\epsilon>0$, we have 
\begin{equation*}
    \sum_{n\leq x} \lambda_{{\rm{sym}}^2 f\otimes f}(n)^2=xQ(\log x)+O_{f,\ \epsilon}\big(x^{(17/18)+\epsilon}\big)
\end{equation*}
where $Q(t)$ is a polynomial of degree $1$.\label{TB}
\end{theoremA}
The aim of this article is to improve Theorems \textbf{A} and \textbf{B}.\\
More precisely we prove:
\begin{theorem}
For any $\epsilon>0$ and $f\in H_k^*$, we have 
\[\sum_{n\leq x}\lambda_{f\otimes f\otimes f}(n)^2=xP(\log x)+O_{f,\ \epsilon}(x^{(695/719)+\epsilon})\] where $P(t)$ is a polynomial of degree $4$.\label{T2.1}
\end{theorem}

\begin{theorem}
For any $\epsilon>0$ and $f\in H_k^*$, we have 
\begin{equation*}
    \sum_{n\leq x} \lambda_{{\rm{sym}}^2f\otimes f}(n)^2=xQ(\log x)+O_{f,\  \epsilon}\big(x^{(2729/2897)+\epsilon}\big)
\end{equation*}
where $Q(t)$ is a polynomial of degree $1$.\label{T2.2}
\end{theorem}

\begin{remark}
\textit{It should be noted that the theorem of K. Ramachandra and A Sankaranarayanan [Lemma 3.5] plays a vital role in the proofs of Theorems {\bf 2.1} and {\bf 2.2}.
It is easy to check that $\frac{695}{719}<\frac{175}{181}$ and $\frac{2729}{2897}<\frac{17}{18}$. Thus Theorem {\bf 2.1} and Theorem {\bf 2.2} are unconditional improvements to the Theorem \textbf{A} and Theorem \textbf{B} respectively. Under the assumption of Lindel{\"o}f Hypothesis the error terms of Theorem {\bf 2.1} and Theorem {\bf 2.2} can slightly be improved for which we refer to section $4$.}
\nonumber
\end{remark}

\section{Lemmas}
\begin{lemma}
Suppose that $\mathfrak{L(s)}$ is a general L-function of degree $m$. Then, for any $\epsilon>0,$ we have \begin{equation}
    \int_T^{2T}|\mathfrak{L}(\sigma+it)|^2  \ dt\ll T^{\max\{m(1-\sigma),1\}+\epsilon}\label{E4}
\end{equation}
uniformly for $\frac{1}{2}\leq \sigma \leq1$ and $T>1$; and \begin{equation}
    \mathfrak{L}(\sigma+it)\ll(|t|+1)^{\frac{m}{2}(1-\sigma)+\epsilon}\label{E5}
\end{equation}
uniformly for $\frac{1}{2}\leq \sigma \leq1+\epsilon$ and $|t|\geq 1$.\label{L3.1}
\end{lemma}
 For some L-functions with small degrees, we invoke either individual or average subconvexity bounds.
\begin{lemma}
For any $\epsilon >0$, we have \begin{equation}
    \int_0^T|\zeta\big(\frac{5}{7}+it\big)|^{12}\ dt\ll_\epsilon T^{1+\epsilon}\label{E6}
\end{equation} uniformly for $T\geq1$.\label{L3.2}
\end{lemma}
\begin{proof}
See, Theorem $8.4$ and $(8.87)$ of \cite{D}.
\end{proof}
\begin{lemma}
For $f\in H_k^*$ and $\epsilon>0$, we have 
\begin{equation}
    L({\rm{sym}}^2 f,\ \sigma+it)\ll_{f,\epsilon}(|t|+1)^{\max\{\frac{11}{8}(1-\sigma),0\}+\epsilon}\label{E7}
\end{equation}
uniformly for $\frac{1}{2}\leq \sigma \leq2$ and $|t|\geq1$.\label{L3.3}
\end{lemma}
\begin{proof}
See, Corollary $1.2$ of \cite{F}.
\end{proof}
\begin{lemma}
For $|t|\geq 10$, we have 
\begin{equation}
    \zeta(\sigma+it)\ll(|t|+10)^{2\kappa(1-\sigma)+\epsilon}\label{E8}
\end{equation}
uniformly for $\frac{1}{2}\leq \sigma \leq 1+\epsilon$ and  for some $\kappa\geq 0$.\label{L3.4}
\end{lemma}
\begin{proof}
Follows from \cite{A} and Maximum modulus principle with $\kappa=\frac{13}{84}$.
\end{proof}

\begin{lemma}[KR+AS]
For $\frac{1}{2}\leq \sigma\leq 2$, T-sufficiently large, there exist a $T^*\in[T,T+T^\frac{1}{3}]$ such that the bound 
\begin{equation}
    \log\zeta(\sigma+iT)\ll (\log\log T^*)^2\ll(\log\log T)^2\label{E9}
\end{equation}
holds uniformly and we have 
\begin{equation}
    |\zeta(\sigma+iT)|\ll \exp((\log\log T)^2)\ll_\epsilon T^\epsilon\label{E10}
\end{equation}
on the horizontal line with $T=T^*$ and $\frac{1}{2}\leq \sigma\leq 2.$\label{L3.5}
\end{lemma}
\begin{proof}
See, Lemma $1$ of \cite{G}.
\end{proof}

\begin{lemma}
For $\Re (s)>1$, define 
\begin{equation*}
L_f(s)=\sum_{n=1}^\infty \frac{\lambda_{f\otimes f\otimes f}(n)^2}{n^s}.\end{equation*}
Then we have 
\begin{equation}
    L_f(s)=\zeta(s)^5L({\rm{sym}}^2f,\ s)^8L({\rm{sym}}^4f,\ s)^4L({\rm{sym}}^4f\otimes {\rm{sym}}^2f,\ s)U(s),\label{E11}
\end{equation}
where the function $U(s)$ is a Derichlet series which converges absolutely for $\Re (s)>\frac{1}{2}$ and $U(s)\neq 0$ for $\Re (s)=1$.\label{L3.6}
\end{lemma}
\begin{proof}
Since we have $\lambda_{f\otimes f\otimes f}(n)^2$ is a multiplicative function and the trivial upper bound $O(n^\epsilon)$, we have that, for $\Re (s)>1$,
\begin{equation*}
L_f(s)=\prod_p\bigg(1+\frac{\lambda_{f\otimes f\otimes f}(p)^2}{p^s}+\frac{\lambda_{f\otimes f\otimes f}(p^2)^2}{p^{2s}}+\ldots\bigg).
\end{equation*}
In the half-plane $\Re (s)>1$, the corresponding coefficients of the term $p^{-s}$ determine the analytic properties of $L_f(s)$. By Lemma 2.1 of \cite{C} we easily find the identity
\begin{align}
    \lambda_{f\otimes f\otimes f}(p)^2 & =(\lambda_{{\rm{sym}}^3f}(p)+2\lambda_f(p))^2\nonumber\\
           & =\lambda_{{\rm{sym}}^3f}(p)^2+4\lambda_{{\rm{sym}}^3f}(p)\lambda_f(p)+4\lambda_f(p)^2.\label{E12}
\end{align}
Consider $\lambda_{{\rm{sym}}^2f\otimes {\rm{sym}}^4f}(p)$, the coefficient of $p^{-s}$ in the Euler product of $L({\rm{sym}}^2f\otimes {\rm{sym}}^4f,\ s)$,
\begin{equation*}
    \lambda_{{\rm{sym}}^2f\otimes {\rm{sym}}^4f}(p)=3+3\alpha_f(p)^2+2\alpha_f(p)^4+\alpha_f(p)^6+3\beta_f(p)^2+2\beta_f(p)^4+\beta_f(p)^6.
\end{equation*}

Now, consider $\lambda_{{\rm{sym}}^3 f}(p)$, the coefficient of $p^{-s}$ in the Euler product of the $L-$function $L({\rm{sym}}^3f,\ s)$,
\begin{equation*}
    \lambda_{{\rm{sym}}^3 f}(p)=\alpha_f(p)^3+\alpha_f(p)+\beta_f(p)^3+\beta_f(p).
\end{equation*}
We have
\begin{align}
    \lambda_{{\rm{sym}}^3 f}(p)^2&=\big(\alpha_f(p)^3+\alpha_f(p)+\beta_f(p)^3+\beta_f(p)\big)^2\nonumber\\
    &=2\big(\alpha_f(p)^2+\beta_f(p)^2+\alpha_f(p)^4+\beta_f(p)^4+2\big)\nonumber
    \\ &\qquad+\alpha_f(p)^2+\beta_f(p)^2+\alpha_f(p)^6+\beta_f(p)^6\nonumber\\
    &=1+\lambda_{{\rm{sym}}^2f\otimes {\rm{sym}}^4f}(p).\label{E13}
\end{align}
Now, consider
\begin{align*}
 \lambda_{{\rm{sym}}^3f}(p)\lambda_f(p)&=\big(\alpha_f(p)^3+\beta_f(p)^3+\alpha_f(p)+\beta_f(p)\big)\big(\alpha_f(p)+\beta_f(p)\big)\\
    &=\alpha_f(p)^4+\beta_f(p)^4+2\alpha_f(p)^2+2\beta_f(p)^2+2.
\end{align*}
By the coefficients of $p^{-s}$ in the Euler products of $L({\rm{sym}}^2f,\ s)$ and $L({\rm{sym}}^4f,\ s)$, we have 
\begin{align}
    \lambda_{{\rm{sym}}^2f}(p)+\lambda_{{\rm{sym}}^4f}(p)&=\alpha_f(p)^4+\beta_f(p)^4+2\alpha_f(p)^2+2\beta_f(p)^2+2\nonumber\\
    &=\lambda_{{\rm{sym}}^3f}(p)\lambda_f(p).\label{E14}
\end{align}

Consider $\lambda_{{\rm{sym}}^2f}(p)$, the coefficient of $p^{-s}$ in the Euler product of $L({\rm{sym}}^2f,\ s)$,
\begin{equation*}
    \lambda_{{\rm{sym}}^2f}(p)=\alpha_f(p)^2+\beta_f(p)^2+1.
\end{equation*}

Now, consider
\begin{align}
    \lambda_f(p)^2&
    =\big(\alpha_f(p)+\beta_f(p)\big)^2\nonumber\\
    &=\alpha_f(p)^2+\beta_f(p)^2+2\nonumber\\
    &=1+\lambda_{{\rm{sym}}^2f}(p).\label{E15}
\end{align}

By using \eqref{E13}, \eqref{E14} and \eqref{E15}, we have
\begin{align*}
    \lambda_{f\otimes f\otimes f}(p)^2&=\big(1+\lambda_{{\rm{sym}}^2f\otimes {\rm{sym}}^4f}(p)\big)+4\big(\lambda_{{\rm{sym}}^2f}(p)+\lambda_{{\rm{sym}}^4}(p)\big)\\
    &\qquad +4\big(1+\lambda_{{\rm{sym}}^2f}(p)\big)\\
    &=5+8\lambda_{{\rm{sym}}^2f}(p)+4\lambda_{{\rm{sym}}^4f}(p)+\lambda_{{\rm{sym}}^2f\otimes {\rm{sym}}^4f}(p).
\end{align*}
Now the lemma follows by standard arguments.
\end{proof}

\begin{lemma}
For $\Re (s)>1$, define \[D_f(s)=\sum_{n=1}^\infty \frac{\lambda_{{\rm{sym}}^2f\otimes f}(n)^2}{n^s}.\]
Then we have \begin{equation}
    D_f(s)=\zeta(s)^2L({\rm{sym}}^2f,\ s)^3L({\rm{sym}}^4f,\ s)^2L({\rm{sym}}^4f\otimes {\rm{sym}}^2f,\ s)V(s),\label{E16}
\end{equation}
where the function $V(s)$ is a Derichlet series which converges absolutely for $\Re (s)>\frac{1}{2}$ and $V(s)\neq 0$ for $\Re (s)=1$.\label{L3.7}
\end{lemma}
\begin{proof}
By $(6.2)$ of \cite{C}, we have 
\begin{equation*}
    \lambda_{{\rm{sym}}^2f\otimes f}(p)=\lambda_{{\rm{sym^3}}f}(p)+\lambda_f(p)
\end{equation*}
for $\Re (s)>1$.

Now the lemma follows in a similar manner as the  proof of Lemma 3.6.
\end{proof}

\begin{proof}[{\bf Proof of Theorem 2.1}]
Firstly, recall that 
\[L_f(s)=\sum_{n=1}^\infty \frac{\lambda_{f\otimes f\otimes f}(n)^2}{n^s}\]
for $\Re (s)>1$ and by \eqref{E11}, we have
\begin{equation*}
    L_f(s)=\zeta(s)^5L({\rm{sym}}^2f,\ s)^8L({\rm{sym}}^4f,\ s)^4L({\rm{sym}}^4f\otimes {\rm{sym}}^2f,\ s)U(s),
\end{equation*}
where the function $U(s)$ is a Derichlet series which converges absolutely for $\Re (s)>\frac{1}{2}$ and $U(s)\neq 0$ for $\Re (s)=1$.\\
By applying Perron formula to $L_f(s)$, we have
\[\sum_{n\leq x}\lambda_{f\otimes f\otimes f}(n)^2=\frac{1}{2\pi i}\int_{b-iT}^{b+iT}L_f(s)\frac{x^s}{s} \ ds+O\bigg(\frac{x^{1+\epsilon}}{T}\bigg)\]
where $b=1+\epsilon$ and $1\leq T\leq x$ is a parameter to be chosen later.\\ 
Now, we make the special choice $T=T^*$ of Lemma 3.5 which satisfies \eqref{E10} and shifting the line of integration to $\Re (s)=\frac{5}{7}$, we have by Cauchy residue theorem
\begin{align}
    \sum_{n\leq x} \lambda_{f\otimes f\otimes f}(n)^2  & =\frac{1}{2\pi i}\bigg\{\int_{\frac{5}{7}-iT}^{\frac{5}{7}-iT}+\int_{\frac{5}{7}+iT}^{b+iT}+\int_{b-iT}^{\frac{5}{7}-iT}\bigg\}L_f(s)\frac{x^s}{s} \ ds\nonumber\\ &\qquad
    +xP(\log x)+O\bigg(\frac{x^{1+\epsilon}}{T}\bigg)\nonumber \\
    & =J_1+J_2+J_3+xP(\log x)+O\bigg(\frac{x^{1+\epsilon}}{T}\bigg)\label{E17}
\end{align}
where $P(t)$ is a polynomial of degree $4$ and the main term $xP(\log x)$ is coming from the residue of $L_f(s)\frac{x^s}{s} $ at the pole $s=1$ of order $5$.\\
For $J_1$, we have 
\begin{align}
    J_1&\ll x^{\frac{5}{7}+\epsilon}\int_1^T\big|\zeta(\frac{5}{7}+it)^5L(\text{{\rm{sym}}}^2f,\ \frac{5}{7}+it)^8L(\text{{\rm{sym}}}^4f,\ \frac{5}{7}+it)^4\nonumber\\
    &\qquad\qquad \qquad L(\text{{\rm{sym}}}^4f\otimes \text{{\rm{sym}}}^2 f,\ \frac{5}{7}+it)\big|t^{-1}  \ dt+x^{\frac{5}{7}+\epsilon}\nonumber\\
&\ll  x^{\frac{5}{7}+\epsilon}\sup_{1\leq T_1\leq T}I_1(T_1)^{5/12}I_2(T_1)^{1/2}I_3(T_1)^{1/12}T_1^{-1},\label{E18}
\end{align}
where
\begin{equation*}
   I_1(T_1)=\int_{T_1}^{2T_1}\big|\zeta\big(\frac{5}{7}+it\big)\big|^{12}  \ dt,
   \end{equation*}
   \begin{equation*}
    I_2(T_1)=\int_{T_1}^{2T_1}\big|L\big(\text{{\rm{sym}}}^2f,\ \frac{5}{7}+it\big)^8L\big(\text{{\rm{sym}}}^4f,\ \frac{5}{7}+it\big)^4\big|^2\ dt
\end{equation*}
 and
 \begin{equation*}
     I_3(T_1)=\int_{T_1}^{2T_1}\big|L\big(\text{{\rm{sym}}}^4f\otimes\text{{\rm{sym}}}^2f,\ \frac{5}{7}+it\big)\big|^{12}\ dt.
 \end{equation*}
 Then by Lemmas 3.1, 3.2 and 3.3, we have
 \begin{equation*}
     I_1(T_1)\ll T_1^{1+\epsilon},\qquad I_3(T_1)\ll 
     T_1^{180/7+\epsilon}
 \end{equation*}
and
\begin{align*}
    I_2(T_1)&\ll T_1^{16\times\frac{11}{8}\times\frac{2}{7}+\epsilon}\int_{T_1}^{2T_1}\big|L\big(\text{{\rm{sym}}}^4f,\ \frac{5}{7}+it\big)^4\big|^2\ dt\\
    &\ll T_1^{12+\epsilon}.
\end{align*}
Hence, we have
\begin{align}
J_1&\ll x^{5/7+\epsilon}\sup_{1\leq T_1\leq T}I_1(T_1)^{5/12}I_2(T_1)^{1/2}I_3(T_1)^{1/12}T_1^{-1}\nonumber\\
    &\ll x^{(5/7)+\epsilon} T^{(635/84)+\epsilon}.\label{E19}
\end{align}
For the integrals over horizontal segments, by using \eqref{E5}, \eqref{E7} and \eqref{E10}, we have 
\begin{align}
    J_2+J_3 & \ll\max_{\frac{5}{7}\leq \sigma\leq b}x^\sigma T^{10\epsilon+\{8\times\frac{11}{8}+\frac{35}{2}\}(1-\sigma)-1}\nonumber \\
    & =T^{10\epsilon}\max_{\frac{5}{7}\leq \sigma\leq b}\bigg(\frac{x}{T^\frac{57}{2}}\bigg)^\sigma T^\frac{55}{2}+\epsilon \nonumber \\
    &\ll x^{\frac{5}{7}+\epsilon}T^{\frac{50}{7}+10\epsilon}+\frac{x^{1+15\epsilon}}{T}.\label{E20}
\end{align}
From \eqref{E17}, \eqref{E19} and \eqref{E20}, we have 
\begin{equation}
    \sum_{n\leq x}\lambda_{f\otimes f\otimes f}(n)^2=xP(\log x)+O\bigg(\frac{x^{1+15\epsilon}}{T}\bigg)+o\bigg(x^{(5/7)+\epsilon}T^{(635/84)+\epsilon}\bigg).\label{E21}
\end{equation}
By taking $T=x^\frac{24}{719}$ in \eqref{E21}, we have 
\begin{equation*}
    \sum_{n\leq x} \lambda_{f\otimes f\otimes f}(n)^2=xP(\log x)+O\big(x^{(695/719)+\epsilon}\big)
\end{equation*}
This completes the proof of Theorem 2.1.
\end{proof}

\begin{proof}
[{\bf Proof of Theorem 2.2}]
Recall that \[D_f(s)=\sum_{n=1}^\infty \frac{\lambda_{{\rm{sym}}^2f\otimes f}(n)^2}{n^s}\]
for $\Re (s)>1$ and by Lemma 3.7, we have 
\begin{equation*}
        D_f(s)=\zeta(s)^2L({\rm{sym}}^2f,\ s)^3L({\rm{sym}}^4f,\ s)^2L({\rm{sym}}^4f\otimes {\rm{sym}}^2f,\ s)V(s),
\end{equation*}
where the function $V(s)$ is a Derichlet series which converges absolutely for $\Re (s)>\frac{1}{2}$ and $V(s)\neq 0$ for $\Re (s)=1$.\\
Now, by applying Perron formula to $D_f(s)$, we have 
\begin{equation*}
    \sum_{n\leq x}\lambda_{{\rm{sym}}^2f\otimes f}(n)^2=\frac{1}{2\pi i}\int_{b-iT}^{b+iT}D_f(s)\frac{x^s}{s} \ ds+O\bigg(\frac{x^{1+\epsilon}}{T
    }\bigg)
\end{equation*}
where $b=1+\epsilon$ and $1\leq T\leq x$ is a parameter to be chosen later.\\
Now, we make the special case $T=T^{**}$ of Lemma 3.5 which satisfies \eqref{E10} and shifting the line of integration to $\Re (s)=\frac{1}{2}+\epsilon$, we have by Cauchy residue theorem 
\begin{align}
    \sum_{n\leq x}\lambda_{{\rm{sym}}^2f\otimes f}(n)^2 &=\frac{1}{2\pi i}\bigg\{\int_{\frac{1}{2}+\epsilon-iT}^{\frac{1}{2}+\epsilon+iT}+\int_{b-iT}^{\frac{1}{2}+\epsilon-iT}+\int_{\frac{1}{2}+\epsilon+iT}^{b+iT} \bigg\}D_f(s)\frac{x^s}{s} \ ds\nonumber\\
    &\qquad+xQ(\log x)+O\bigg(\frac{x^{1+\epsilon}}{T}\bigg)\nonumber\\
    &=J_1+J_2+J_3+xQ(\log x)+O\bigg(\frac{x^{1+\epsilon}}{T}\bigg).\label{E22}
\end{align}
Where $Q(t)$ is a polynomial of degree $1$ and the main term $xQ(\log x)$ is coming from the residue of $D_f(s)\frac{x^s}{s}$ at the pole $s=1 $ of order $2$.\\
For $J_1$, we have 
\begin{multline}
J_1\ll x^{\frac{1}{2}+\epsilon}\int_1^T\big|\zeta(\frac{1}{2}+\epsilon+it)^2L({\rm{sym}}^2f,\ \frac{1}{2}+\epsilon+it)^3L({\rm{sym}}^4f,\ \frac{1}{2}+\epsilon+it)^2\\L({\rm{sym}}^4f\otimes {\rm{sym}}^2f,\ \frac{1}{2}+\epsilon+it)V(\frac{1}{2}+\epsilon+it)\big|t^{-1}\ dt+x^{\frac{1}{2}+\epsilon}.\label{E23}
\end{multline}
By Lemma 3.4 and Cauchy-Schwarz inequality, we have
\begin{multline}
J_1\ll x^{\frac{1}{2}+\epsilon}\sup_{1\leq T_1\leq T}T_1^{2\times2\kappa\times\frac{1}{2}+\frac{33}{16}}\bigg(\int_{T_1}^{2T_1}\big|L({\rm{sym}}^4f,\ \frac{1}{2}+it)^2\big|^2  \ dt\bigg)^\frac{1}{2}\\
\bigg(\int_{T_1}^{2T_1}\big|L({\rm{sym}}^4f\otimes {\rm{sym}}^2f,\ \frac{1}{2}+it)\big|^2  \ dt\bigg)^\frac{1}{2}T_1^{-1}+x^{\frac{1}{2}+\epsilon}.\nonumber
\end{multline}
By \eqref{E4} of Lemma 3.1, we have
\begin{align}
    J_1 & \ll x^{\frac{1}{2}+\epsilon}+x^{\frac{1}{2}+\epsilon}\sup_{1\leq T_1\leq T}T_1^{2\kappa+\frac{33}{16}+\frac{10}{2}\times\frac{1}{2}+\frac{15}{2}\times\frac{1}{2}-1+\epsilon}\nonumber\\  &\ll x^{(1/2)+\epsilon}T^{2\kappa+(117/16)+\epsilon}.\label{E24}
\end{align}
For the integrals over the horizontal segments, by using \eqref{E5}, \eqref{E7} and \eqref{E10}, we have 
\begin{align}
    J_2+J_3&\ll \max_{\frac{1}{2}+\epsilon\leq\sigma\leq b}x^\sigma T^{\{3\times\frac{11}{8}+\frac{25}{2}\}(1-\sigma)-1+2\epsilon}\nonumber\\
    &=T^{2\epsilon}\max_{\frac{1}{2}+\epsilon\leq\sigma\leq b}\bigg(\frac{x}{T^\frac{133}{8}}\bigg)^\sigma T^\frac{125}{8}\nonumber\\
    &\ll x^{(1/2)+\epsilon}T^{(117/16)+2\epsilon}+\frac{x^{1+10\epsilon}}{T}.\label{E25}
\end{align}
From \eqref{E22}, \eqref{E24} and \eqref{E25}, we have
\begin{equation}
\sum_{n\leq x}\lambda_{{\rm{sym}}^2f\otimes f}(n)^2\ll xQ(\log x)+x^{(1/2)+\epsilon}T^{2\kappa+(117/16)+\epsilon}+O\bigg(\frac{x^{1+\epsilon}}{T}\bigg).\label{E26}
\end{equation}
By taking $T=x^{\frac{8}{133+32\kappa}}$ in \eqref{E26}, we have
\begin{equation}
    \sum_{n\leq x} \lambda_{{\rm{sym}}^2f\otimes f}(n)^2=xQ(\log x)+O(x^{\frac{125+32\kappa}{133+32\kappa}+\epsilon})\label{E27}.
\end{equation}
Now Theorem {\bf 2.2} follows by taking $\kappa=\frac{13}{84}$ and we obtain
\begin{equation*}
\sum_{n\leq x} \lambda_{{\rm{sym}}^2f\otimes f}(n)^2=xQ(\log x)+O(x^{(2729/2897)+\epsilon}).
\end{equation*}
\end{proof}

\section{Some conditional results}
\textbf{\underline{Lindel{\"o}f Hypothesis for $\zeta(s):$}}\\
 This states that for $|t|\geq 10$
\begin{equation}
    \zeta(\sigma+it)\ll(|t|+10)^\epsilon\label{E28}
\end{equation}
for all $\epsilon>0$ uniformly for $\frac{1}{2}\leq\sigma\leq 2$.
\lbrack Refer \cite{H}, pp. 328-335\rbrack.
\begin{theorem}
Assuming Lindel{\"o}f Hypothesis for $\zeta(s)$. For any $\epsilon>0$ and $f\in H_k^*$, we have 
\[\sum_{n\leq x}\lambda_{f\otimes f\otimes f}(n)^2=xP(\log x)+O_{f,\ \epsilon}(x^{(55/57)+\epsilon})\] where $P(t)$ is a polynomial of degree $4$.\label{T4.1}
\end{theorem}
\begin{theorem}
Assuming Lindel{\"o}f Hypothesis for $\zeta(s)$. For any $\epsilon>0$ and $f\in H_k^*$, we have 
\begin{equation*} 
    \sum_{n\leq x} \lambda_{{\rm{sym}}^2f\otimes f}(n)^2=xQ(\log x)+O_{f,\ \epsilon}\big(x^{(125/133)+\epsilon}\big)
\end{equation*}
where $Q(t)$ is a polynomial of degree $1$.\label{T3.2}
\end{theorem}

\begin{proof}[{\bf Proof of Theorem 4.1}]
From the proof of Theorem {\bf 2.1}, recall that 
\begin{align}
    \sum_{n\leq x} \lambda_{f\otimes f\otimes f}(n)^2 =J_1+J_2+J_3+xP(\log x)+O\bigg(\frac{x^{1+\epsilon}}{T}\bigg)\label{E29}
\end{align}
Where $P(t)$ is a polynomial of degree $4$.

For $J_1$ by \eqref{E7}, \eqref{E28} and Cauchy-Schwarz inequality, we have
\begin{align*}
    J_1&\ll x^{\frac{5}{7}+\epsilon}+ x^{\frac{5}{7}+\epsilon}\int_1^T\big|\zeta(\frac{5}{7}+\epsilon+it)^5L(\text{{\rm{sym}}}^2f,\ \frac{5}{7}+\epsilon+it)^8\\
    &\qquad L(\text{{\rm{sym}}}^4f,\ \frac{5}{7}+\epsilon+it)^4 L(\text{{\rm{sym}}}^4f\otimes \text{{\rm{sym}}}^2 f,\ \frac{5}{7}+\epsilon+it)\big|t^{-1}  \ dt\\ &\ll x^{\frac{5}{7}+\epsilon}+x^{\frac{5}{7}+\epsilon}\\&\qquad\sup_{1\leq T_1\leq T}\bigg\{\bigg\{\max_{T_1\leq t\leq 2T_1}\bigg|\frac{\zeta(\frac{5}{7}+\epsilon+it)^5L(\text{{\rm{sym}}}^2f,\ \frac{5}{7}+\epsilon+it)^8}{t}\bigg|\bigg\}\\
&\qquad\bigg\{\int_{T_1}^{2T_1}|L(\text{{\rm{sym}}}^4f,\ \frac{5}{7}+\epsilon+it)^4L(\text{{\rm{sym}}}^4f\otimes \text{{\rm{sym}}}^2 f,\ \frac{5}{7}+\epsilon+it)|   \ dt\bigg\}\bigg\}\\
&\ll x^{\frac{5}{7}+\epsilon}+x^{\frac{5}{7}+\epsilon}\sup_{1\leq T_1\leq T}T_1^{5\epsilon+\frac{22}{7}-1}\bigg(\int_{T_1}^{2T_1}\big|L({\rm{sym}}^4f,\ \frac{5}{7}+\epsilon+it)^4\big|^2  \ dt \bigg)^\frac{1}{2}\\
&\qquad \bigg(\int_{T_1}^{2T_1}\big|L({\rm{sym}}^4f\otimes{\rm{sym}}^2f,\ \frac{5}{7}+\epsilon+it)\big|^2  \ dt \bigg)^\frac{1}{2}.
\end{align*}
By \eqref{E4}, we have
\begin{align*}
    J_1&\ll x^{\frac{5}{7}+\epsilon}+x^{\frac{5}{7}+\epsilon}\sup_{1\leq T_1\leq T}T_1^{5\epsilon+\frac{22}{7}+\frac{20}{2}\times\frac{2}{7}+\frac{15}{2}\times\frac{2}{7}-1}\\
    &\ll x^{(5/7)+\epsilon}T^{(50/7)+10\epsilon}.\
\end{align*}
For the integrals over horizontal segments, by \eqref{E20}, we have 
\begin{align*}
    J_2+J_3\ll x^{(5/7)+\epsilon}T^{(50/7)+10\epsilon}+\frac{x^{1+15\epsilon}}{T}.
\end{align*}
Hence by \eqref{E29}, we have
\begin{equation}
    \sum_{n\leq x}\lambda_{f\otimes f\otimes f}(n)^2=xP(\log x)+O\bigg(\frac{x^{1+15\epsilon}}{T}\bigg)+O\bigg(x^{(5/7)+\epsilon}T^{(50/7)+\epsilon}\bigg).\label{E30}
\end{equation}
By taking $T=x^\frac{2}{57}$ in \eqref{E30}, we have 
\begin{equation*}
    \sum_{n\leq x} \lambda_{f\otimes f\otimes f}(n)^2=xP(\log x)+O\big(x^{(55/57)+\epsilon}\big).
\end{equation*}
\end{proof}

\begin{proof}[{\bf Proof of Theorem 4.2}]
From the asymptotic formula in \eqref{E27}, by assuming Lindel{\"o}f Hypothesis for $\zeta(s)$, we have  $\kappa=\epsilon$, where $\epsilon$ is any positive constant and we obtain
\begin{equation*}
\sum_{n\leq x} \lambda_{{\rm{sym}}^2f\otimes f}(n)^2=xQ(\log x)+O\big(x^{(125/133)+\epsilon}\big).
\end{equation*}
\end{proof}
\begin{remark}
It is easy to check that $\frac{55}{57}<\frac{695}{719}<\frac{175}{181}$ and $\frac{125}{133}<\frac{2729}{2897}<\frac{17}{18}$ (pertain into Theorem \textbf{A} and Theorem \textbf{B}).
\end{remark}

\begin{concluding remark} If one has the Lindel{\"o}f Hypothesis bound for the $L-$ function $L_f(s)$, namely 
\begin{equation*}
    L_f(\sigma+it)\ll (|t|+10)^\epsilon 
\end{equation*}
holds for all $\epsilon>0$ uniformly for $\frac{1}{2}\leq\sigma\leq 2$ and $|t|\geq 10$, then it is not difficult to see that the asymptotic formula
\begin{equation*}
    \sum_{n\leq x}\lambda_{f\otimes f\otimes f}(n)^2=xP(\log x)+O(x^{\frac{1}{2}+\epsilon})
\end{equation*} holds, where $P(t)$ is a polynomial of degree $4
$. Of course such an expected improvement is far away.
\end{concluding remark}

\subsection*{Acknowledgements}
The first author wishes to express his thankfulness to the Funding Agency "Ministry of Human Resource Development (MHRD), Govt. of India" for the fellowship PMRF, Appl. No.PMRF-2122-3190   for it's financial support.

\end{document}